\documentclass[a4]{article}
\usepackage{amsmath,amsthm,amsfonts,amssymb,amscd,epsf,epsfig,psfrag,enumerate, tikz}
\usepackage[margin=2.5cm]{geometry}

\def\conv{\operatorname{conv}}
\def\intt{\operatorname{int}}
\def\aff{\operatorname{aff}}
\def\vol{\operatorname{vol}}
\def\R{\mathbb{R}}
\def\Z{\mathbb{Z}}
\def\N{\mathbb{N}}
\def\Q{\mathbb{Q}}

\newcommand{\comment}[1]{}

\newtheorem{theorem}{Theorem}
\numberwithin{theorem}{section}
\newtheorem{lemma}[theorem]{Lemma}

\newtheorem{remark}[theorem]{Remark}
\newtheorem{corollary}[theorem]{Corollary}

\title{Convex integer minimization in fixed dimension}

\author{Timm Oertel, Christian Wagner, Robert Weismantel}

\begin{document}

\maketitle

\begin{abstract}
We show that minimizing a convex function over the integer
points of a bounded convex set is polynomial in fixed dimension.
\end{abstract}


\section{Introduction.}

One of the most important complexity results in integer programming
states that minimizing a linear function over the integer
points in a polyhedron is solvable in polynomial time
provided that the number of integer variables is a constant.
This landmark result due to Lenstra \cite{Lenstra83} has
been generalized by Barvinok \cite{Barvinok94}: he shows
that one can even count the number of integer points in polytopes in
fixed dimension.
More recent extensions of Lenstra's algorithm apply to
integer optimization problems associated with semi-algebraic sets  and
described by quasi-convex polynomials.
The first polynomial time algorithm for minimizing a
quasi-convex polynomial over such sets in fixed dimension is
due to Khachiyan and Porkolab \cite{KhachiyanPorkolab00}.
Recent improvements of the complexity bound are given by Heinz
\cite{Heinz05} and by Hildebrand  and
K\"oppe \cite{HildebrandKoeppe12}.

In this paper, we drop the assumption that the functions describing the
input to our problem are polynomials.
Instead, we aim at minimizing a general convex or quasi-convex
function over the integer points in a bounded convex set in fixed
dimension.
The bounded convex set is defined by convex or quasi-convex functions as well.
We assume that all the functions are encoded by means of
evaluation oracles: queried on a rational point, the
evaluation oracles return the function values that the point
attains. 
We assume that three further oracles are given, namely a continuous feasibility oracle,
a separating hyperplane oracle and a linear integer optimization oracle.
In order to realize them one needs additional assumptions on
the functions.

It is well known that  Lenstra's algorithm can in principle be applied
to any class of convex sets ${\cal C}$ in $\R^n$ when $n$ is a constant
 provided
that we can determine an ellipsoidal approximation for every
member in ${\cal C}$ efficiently.
By an \emph{ellipsoidal approximation} of a convex set we
mean an ellipsoid $E$ contained in the convex set such that
a properly scaled version of $E$ contains the convex set.
(Typically, the scaling factor is ${\cal O}(n)$).
The construction of such an ellipsoidal approximation is
often performed by designing a
\emph{shallow cut separation oracle}
(see, for instance,
\cite[Section 3.3]{GroetschelLovaszSchrijver-Book88}).
To the best of our knowledge it is not known how to
construct ellipsoidal approximations for general convex sets
in polynomial time even when the number of variables is
fixed. This explains why general convex integer minimization problems
with a fixed number of variables have not yet been extensively studied.

We  design a novel polynomial time algorithm for
general convex integer minimization problems in fixed
dimension that avoids going through the construction of
ellipsoidal approximations.
Instead we develop a cone-shrinking algorithm that from
iteration to iteration produces smaller and smaller cones
containing the convex set under consideration until we can
reduce the original question to a series of similar problems
in smaller dimensions.

Our assumptions are as follows.
Let $f_0, \ldots,f_m : \R^n \mapsto \R$ be  quasi-convex functions,
i.e.~for every $\alpha \in \R$ the level set
$\{x \in \R^n \; | \; f_i(x) \le \alpha\}$ is convex.
Note that a convex function is also quasi-convex.
For a given $\varepsilon\in\R_{>0}$, we define
\begin{flalign*}
K_0&:=\left\lbrace x\in \R^n\;|\;f_i(x)\le 0\text{ for all
}i=0,\dots,m\right\rbrace,\\
K_\varepsilon&:=\left\lbrace x\in \R^n\;|\;f_i(x)\le \varepsilon\text{
for all }i=0,\dots,m\right\rbrace.
\end{flalign*}
Moreover, let $B,\Delta,M \in \N$ be given numbers.
We assume that $K_0,\;K_\varepsilon\subset[-B,B]^n$ and that $|f_i(x)|\le M$ for all $x\in[-B,B]^n$ and all $i=0,...,m$.
For a point $x \in \Q^n$ the precision of $x$
is the smallest integer $q \in \N$ such that $x$ has a
representation
$x = ( \frac{p_1}{q}, \ldots, \frac{p_n}{q} )$,
where $p_j \in \Z$ for all $j = 1, \ldots, n$.
We are interested in rational points with a precision of at most $\Delta$.
We assume to have available three oracles.

\begin{quote}
  {\bf $\Delta$-Feasibility Oracle.}
    Given a polytope $P$ in $[-B,B]^n$, the oracle
    returns a point $x \in P\cap K_{\frac{\varepsilon}{2}}$
    with precision at most $\Delta$, or certifies that
    $P\cap K_0$ does not contain a point with precision at
    most $\Delta$. \\[2mm]
  {\bf Separating Hyperplane Oracle.}
    Given an affine space $A$ and a point $a\in A$,
    the oracle returns either that $a\in K_\varepsilon\cap A$
    or it returns a vector $c\in\Q^n$ with $\|c\|_\infty=1$ such that
    $c^\mathsf{T} x\le c^\mathsf{T} a$ for every
    $x\in K_{\frac{\varepsilon}{2}}\cap A$. \\[2mm]
  {\bf Linear Integer Optimization Oracle.}
    Given a polytope $P$ and a linear objective function,
    the oracle returns a point in $P \cap \Z^n$
    with minimum objective function value, or certifies
    that $P \cap \Z^n$ is empty.
\end{quote}

The polytope $P$ that is part of the input in the $\Delta$-Feasibility Oracle
will always be defined as the intersection of the box $[-B,B]^n$ with an
affine space. 
Polynomial time algorithms for realizing a \mbox{$\Delta$-Feasibility} Oracle
can be found in \cite{Ben-TalNemirovski2001} and
\cite{Nesterov-Book04}.

To the best of our knowledge there is no efficient algorithm for
realizing a Separating Hyperplane Oracle in general.
Rather, concrete realizations depend on properties of the functions
$f_i$, $i = 0, \dots, m$.
One particularly relevant case in which the
Separating Hyperplane Oracle can be emulated is as follows.
Let us assume that the functions $f_0, \ldots, f_m$ are
convex. Moreover, let us assume that, for every $x \in [-B,B]^n$ and for every $i\in \{0,\dots, m\}$, a subgradient of $\partial f_i(x)$ is known.
Suppose now that an affine space $A$ and a point $a \in A$ are
given.
The  question is to decide whether
$a \in K_\varepsilon \cap A$, or -- if not -- to
find a hyperplane that separates $a$ from
$K_\frac{\varepsilon}{2} \cap A$.
We start by checking whether
$a \in K_\varepsilon \cap A$.
This can be done by simply substituting $a$ into the
functions $f_i$, $i=0,\ldots,m$.
Let us assume that $a \notin K_\varepsilon \cap A$.
Then there exists one of the
functions $f_i$, say $f_0$, such that
$f_0(a) > \varepsilon > \frac{\varepsilon}{2}$.
We take an element $g\in\partial f_0(a)$.
Note that $g$ is the normal vector of a tangent hyperplane $H$ of the epigraph of $f_0$ at the point $(a,f_0(a))$.
Next we shift $H$ such that it contains the point $a$.
Let the resulting hyperplane be $H'$.
Then $H' \cap A$ is a separating hyperplane.

For a realization of the Linear Integer Optimization Oracle
we refer again to the paper of Lenstra \cite{Lenstra83}.

We emphasize that the parameter $\varepsilon$
does not affect the number of iterations of our cone-shrinking
algorithm.
In fact, from now on we assume that $\varepsilon$ is
fixed.
Of course, it plays a role in the realizations of our
oracles.
Our main contribution is stated in the theorem below.

\begin{theorem}\label{thm:ConIntMin}
Let $f_0, \ldots,f_m$, $K_0$,
$K_\varepsilon$, $B$, and $\Delta$ be
as above.
In polynomial time in $\log(B)$
and $\log(\Delta)$ either one can find a point
$z \in K_\varepsilon \cap \Z^n$, or show that
$K_0 \cap \Z^n = \emptyset$.
\end{theorem}

Theorem~\ref{thm:ConIntMin} allows us to minimize any
quasi-convex function $f_0$ over the integer points of a
convex set
$L := \{x \in [-B,B]^n \; | \; f_i(x) \le 0,
\text{ for all } i = 1, \dots, m\}$
in polynomial time, when $n$ is fixed.
An approximate solution of the problem
$\min \{f_0(x) \; | \; x \in L \cap \Z^n \}$ can be computed
by binary search in the interval $[-M,M]$.
This follows since, for any $\gamma \in [-M,M]$,
Theorem~\ref{thm:ConIntMin} can be applied to the set
$\{x \in [-B,B]^n \; | \; f_0(x) - \gamma \le 0, \;
f_i(x) \le 0, \text{ for all } i = 1, \dots, m\}$ instead of
$K_0$.
We thus derive the following corollary.

\begin{corollary}
Let $f_0, \ldots,f_m$, $B$,
$\Delta$, and $M$ be as above.
In polynomial time in $\log(B)$, $\log(\Delta)$ and 
$\log(M)$ either one can find a point $z \in \Z^n$ such
that $f_0(z) \le \min \{f_0(x) \; | \; x \in \Z^n
\text{ and } f_i(x) \le \varepsilon \text{ for all }
i= 1, \dots, m\} + \varepsilon$, or show that the problem
$\min \{f_0(x) \; | \; x \in \Z^n \text{ and } f_i(x) \le
0 \text{ for all } i= 1, \dots, m\}$ is infeasible.
\end{corollary}

In the next section, we introduce the notation and we prove
statements that are needed to show
Theorem~\ref{thm:ConIntMin}.
Section~\ref{sec.thm} contains the proof of
Theorem~\ref{thm:ConIntMin}.
Section~\ref{sec:mixed-integer} describes a straightforward
generalization of our algorithm to the mixed integer case.

\section{Auxiliary lemmata.}

For a set $S \subset \R^n$ we denote by $\aff(S)$ the affine
hull of $S$, by $\conv(S)$ the convex hull
of $S$, by $\intt(S)$ the interior of $S$, by $\dim(S)$ the
dimension of the smallest affine space containing $S$, and
by $\vol_j(S)$, $j = 1, \dots, n$, the Lebesgue measure of
$S$ with respect to a $j$-dimensional affine subspace
containing it.
We omit the subscript and simply write $\vol(S)$ whenever
the dimension is clear from the context.
For two sets $S,T \subset \R^n$ we denote by
$S+T := \{s+t: s \in S,\ t \in T\}$ the Minkowski sum of $S$
and $T$.
When $S$ is bounded, the set $S-S := \{x-y : x,y \in S\}$ is
called the difference body of $S$.
If $M \in \R^{n \times n}$ is a matrix, then $\det(M)$
denotes the determinant of $M$.

In the remainder of this section we present five lemmata.
Lemmata~\ref{shrinking} and \ref{covering} are needed to
prove Lemma~\ref{covering2}.
Lemmata~\ref{minimizing}, \ref{cutting}, and \ref{covering2}
are used in the next section to prove
Theorem~\ref{thm:ConIntMin}.
The following lemma states that the convex hull of the
integer points of an $n$-dimensional closed convex set is
lower-dimensional whenever its volume is sufficiently
small. 

\begin{lemma}\label{minimizing}
Let $K \subset \R^n$ be a closed convex set such that
$\vol(K) < \frac{1}{n!}$.
Then
$\dim(\conv(K \cap \Z^n))\le n-1$.
\end{lemma}

\begin{proof}
For the purpose of deriving a contradiction assume that
there exist $n+1$ affinely independent points
$v_0, \dots ,v_n\in K\cap\Z^n$.
Then
$\vol(\conv(\{v_0, \dots, v_n\}))=\frac{1}{n!} \mathopen |\det(v_1-v_0,\dots,v_n-v_0)|\ge\frac{1}{n!}$.
\end{proof}

In the subsequent lemma we define for every $n$-dimensional
closed convex set $K$ a corresponding set
$\bar{K}$ that is contained in $K$.
The set $\bar{K}$ has the property that the Minkowski sum of
$\bar{K}$ and a certain scaling of the difference body of
$K$ is a subset of $K$.
Then, Lemma~\ref{shrinking} gives an outer approximation of
$K$ and an inner approximation of $\bar{K}$ in terms of a
certain ellipsoid.
Consequently, this ellipsoid can be used to approximate
$K\setminus\bar{K}$.
Furthermore, we always have that $\bar{K}\neq\emptyset$.

\begin{lemma}\label{shrinking}
Let $K\subset\R^n$ be an $n$-dimensional closed convex set,
and let
\begin{equation*}
\bar{K} := \left\lbrace x\in\R^n\;\Big{|}\;x+\frac{1}{4n}(K-K)\subset K\right\rbrace.
\end{equation*}
Then there exists an ellipsoid $E \subset \R^n$ and a point
$c \in \bar{K}$ such that
$c + \frac{1}{2} E \subset \bar{K}$ and
$K \subset c + n E$.
\end{lemma}

\begin{proof}
By John's characterization of inscribed ellipsoids of
maximal volume (see John \cite{John48} and Ball
\cite{Ball92}), there exists an ellipsoid $E$ centered at
the origin, and a point
$c$ such that $c + E \subset K \subset c + n E$.
By the definition of the difference body $K-K$, it follows
that
$2E=E-E\subset K-K\subset nE-nE= 2nE$.
This implies $\frac{1}{4n}(K-K)\subset\frac{1}{2}E$ and thus
$\frac{1}{2}E+\frac{1}{4n}(K-K)\subset E \subset K - c$.
Hence, $(c + \frac{1}{2} E) + \frac{1}{4n}(K-K)
\subset K$.
This implies that $c + \frac{1}{2} E \subset \bar{K}$.
In particular, $c \in \bar{K}$.
\end{proof}

\begin{remark}\label{rem:shrinking}
If $K\subset\R^n$ is a polytope, then the set $\bar{K}$ as defined in the previous lemma can be computed explicitly.
For that, let $K= \{ x\in\R^n \;|\; a_i^\mathsf{T}x\le b_i, \text{ for } i=1,\ldots, m \}$ be represented by facet-defining inequalities.
Then, for all $i=1,\ldots,m$, we define $\rho_i:=b_i-\min\{a_i^\mathsf{T}x\;|\;x\in K\}$, 
i.e. the width of $K$ with respect to $a_i$.
Since for the difference body $K-K$ it holds that
$\max\{a_i^\mathsf{T}x\;|\;x\in K-K\}-\min\{a_i^\mathsf{T}x\;|\;x\in K-K\}=2\rho_i$ for all $i$,
it follows that $\bar{K}=\{ x\in\R^n \;|\; a_i^\mathsf{T}x\le b_i-\frac{1}{4n}\rho_i, \text{ for } i=1,\ldots, m \}$.
\end{remark}

For the two sets $K$ and $\bar{K}$ defined in 
Lemma~\ref{shrinking}, we prove next that when
intersecting $K$ with a half-space containing a point of
$\bar{K}$ on its boundary, the volume of this
intersection is guaranteed to decrease by a constant factor
that is only dependent on the dimension.

\begin{lemma}\label{cutting}
Let $K\subset\R^n$ be an $n$-dimensional closed convex set,
and let
$\bar{K}$ be defined as in Lemma~\ref{shrinking}. 
Furthermore, let $x^\star\in\bar{K}$ and let $H^+$ be a half-space with $x^\star$ on its boundary.
Then
\begin{equation*}
\vol(K\cap H^+)\le\left(1-\frac{1}{n^n2^{n+1}}\right)\vol(K).
\end{equation*}
\end{lemma}

\begin{proof}
The Brunn-Minkowski inequality
(see, for instance, Gruber \cite[Theorem 8.5]{Gruber-Book07}) states that $2^n\vol(K)\le\vol(K-K)$.
In addition, we have
$x^\star+\frac{1}{4n}(K-K)\subset K$.
Furthermore, due to the central symmetry of the difference
body $K-K$, we have
\begin{equation*}
\vol\left(\left(x^\star+\frac{1}{4n}(K-K)\right)\cap H^+\right)=\frac{1}{2}\vol\left(\frac{1}{4n}(K-K)\right).
\end{equation*}
Hence,
\begin{equation*}
\vol(K\cap H^+)\le\vol(K)-\frac{1}{2}\vol\left(\frac{1}{4n}(K-K)\right)\le\left(1-\frac{1}{n^n2^{n+1}}\right)\vol(K).
\end{equation*}
\end{proof}

The following lemma is one of the key ingredients of our
proof of Theorem~\ref{thm:ConIntMin}.
It applies to two similar truncated second order cones:
if one of them does not contain a point of a
lattice, then the lattice points contained in the other truncated
cone lie on a number of hyperplanes which only depends on $n$.

\begin{lemma}\label{covering}
Let $\Lambda$ be an arbitrary lattice in $\R^n$.
Moreover, let
\begin{equation*}
C:=\left\lbrace x\in\R^n\;\Big{|}\;\frac{1}{2(n-1)}\sum_{i=1}^{n-1}x_i^2\le x_n \le1\right\rbrace,
\end{equation*}
and let
\begin{equation*}
\bar{C}:=\left\lbrace  x\in\R^n\;\Big{|}\;\sum_{i=1}^{n-1}x_i^2\le x_n\le1\right\rbrace.
\end{equation*}
If $\intt(\bar{C})\cap\Lambda=\emptyset$, then the lattice points $C\cap\Lambda$ lie on at most $4^n n^{3n}$ hyperplanes.
\end{lemma}

\begin{proof}
Our idea is to cover $C$ with $4^n n^{3n}$ boxes.
Then we show that the lattice points in each box lie on a
single hyperplane.
We note that, if a convex set $L \subset \R^n$ satisfies
$L + \Lambda = \R^n$, then any translate of $L$ contains at
least one point of $\Lambda$.
Furthermore, observe that, for any points
$v_0, \dots, v_n \in [0,\frac{1}{n^2}]^n$, it holds that
\begin{equation}\label{coveringArg1}
\left\lbrace x\in \R^n \;{\Big |}\; x=\sum_{i=1}^n
 \lambda_i(v_i-v_0) \text{ and }
 -\frac{1}{2}\le\lambda_i\le\frac{1}{2} \text{ for all } i
 = 1, \dots, n\right\rbrace\subset\left[-\frac{1}{2n},\frac{1}{2n}\right]^n\end{equation}
and for a sufficiently small $\alpha>0$ it holds
\begin{equation}\label{coveringArg2}
\left[-\frac{1}{2n},\frac{1}{2n}\right]^{n-1}\times\left[\frac{n-1-\alpha}{n},\frac{n-\alpha}{n}\right]\subset\intt(\bar{C}).
\end{equation}
It is straightforward to check that the right hand side in
\eqref{coveringArg1} and the left hand side in
\eqref{coveringArg2} are translates.
More precisely, the set $[-\frac{1}{2n}, \frac{1}{2n}]^n +
\frac{2(n-\alpha)-1}{2n} e_n =
[-\frac{1}{2n}, \frac{1}{2n}]^{n-1} \times
[\frac{n-1-\alpha}{n}, \frac{n-\alpha}{n}]$, where $e_n$
denotes the $n$-th unit vector.
Observe that $C \subset [-2n,2n)^n$.
Next we partition $[-2n,2n)^n$ into boxes.
Let $D := [-2n,2n)^n \cap \frac{1}{n^2} \Z^n$.
Then the cardinality of $D$ is $4^n n^{3n}$.
Moreover, $C\subset [-2n,2n)^n \subset D+[0,\frac{1}{n^2}]^n$.
Now assume that there exists a box
$d + [0,\frac{1}{n^2}]^n$, with $d \in D$, that contains
$n+1$ affinely independent lattice points $v_0, \dots, v_n$,
i.e.~assume that
$v_0, \dots, v_n \in \Lambda \cap ( d+[0,\frac{1}{n^2}]^n )$.
Then $\lbrace x\in \R^n \; | \; x=\sum_{i=1}^n
 \lambda_i(v_i-v_0) \text{ and }
 -\frac{1}{2}\le\lambda_i\le\frac{1}{2} \text{ for all } i
 = 1, \dots, n\rbrace + \Lambda = \R^n$.
This, together with (\ref{coveringArg1}) and (\ref{coveringArg2}),
contradicts $\intt(\bar{C})\cap\Lambda=\emptyset$.
\end{proof}

In order to apply Lemma~\ref{covering} in our proof of
Theorem~\ref{thm:ConIntMin}, we will adapt it to the
notation that will be used later and we will show that we
can compute the hyperplanes efficiently. 

\begin{lemma}\label{covering2}
Let $\Lambda$ be an arbitrary lattice.
Let $P\subset\R^n$ be a $(n-1)$-dimensional polytope, and let
$\bar{P}$ be defined as in Lemma~\ref{shrinking}.
Furthermore, let $y \in \R^n \setminus \aff(P)$ such that
$\intt\left(\conv(\{y\},\bar{P})\right) \cap \Lambda = \emptyset$.
In polynomial time in the input size of $P$ and $y$, we can construct 
hyperplanes containing all the lattice points $\conv(\{y\},P)\cap\Lambda$.
The number of hyperplanes is at most $4^nn^{3n}$.
\end{lemma}

\begin{proof}
From Lemma \ref{shrinking}, it follows that there exists a $(n-1)$-dimensional ellipsoid $E$ and a point $c\in\bar{P}$ such that
$c+E\subset\bar{P}$ and $P\subset c+2(n-1)E$.
In particular we can compute $\bar{P}$ in polynomial time (see Remark \ref{rem:shrinking}) and hence we can compute $E$ and $c$ in polynomial time (see \cite[Theorem 3.3.3]{GroetschelLovaszSchrijver-Book88} or \cite{VandenbergheBoydWu1998}).
Moreover, there exists a bijective affine mapping $A:\R^n\mapsto\R^n$
such that
$A\left(\conv\left(\{y\},c+2(n-1)E\right)\right)=C$
and
$A\left(\conv\left(\{y\},c+E\right)\right)=\bar{C}$,
with $C$ and $\bar{C}$ as in Lemma~\ref{covering}.
By applying Lemma~\ref{covering} to  $C$, $\bar{C}$, and
$\Lambda^\prime = A(\Lambda)$ it follows that we can place the lattice points $\conv(\{y\},P)\cap\Lambda$
onto at most $4^n n^{3n}$ hyperplanes.

It remains to construct the hyperplanes.
For that, we use again the notation of the previous Lemma~\ref{covering} and its proof.
For every $d\in D$ there exists a hyperplane $H_d$ such that $A^{-1}(d+[0,\frac{1}{n^2}]^n)\cap\Lambda\subset H_d$.
By  \cite[Lemma 6.5.3]{GroetschelLovaszSchrijver-Book88}, we can determine $H_d$ explicitly in polynomial time.
\end{proof}

\section{Proof of Theorem \ref{thm:ConIntMin}.} \label{sec.thm}

Let us first outline the main steps of the proof.

We start by applying the $\Delta$-Feasibility Oracle to
the polytope $[-B,B]^n$.
Assume that the oracle returns a point $y \in K_{\frac{\varepsilon}{2}}$.
We then consider an arbitrary facet $F$ of $[-B,B]^n$, and define
the set $T_0 := \conv(\{y\},F)$.
The basic idea is to successively construct
subsets $T_0 \supset T_1 \supset T_2 \ \dots$
that satisfy $\vol(T_{i+1}) < \vol(T_i)$ for all $i$.
This subset construction is iterated until we either obtain a
set $T_i$ in which we can find 
an integer point $z\in K_\varepsilon$; or the volume of one of the constructed sets 
is so small that we can apply
Lemma~\ref{minimizing} to reduce the $n$-dimensional problem
to a $(n-1)$-dimensional problem.
By our hypothesis of induction, the $(n-1)$-dimensional problem can be solved
in polynomial time.

Let us now explain how the construction of the sets $T_i$ is
implemented. 
Each set $T_{i+1}$ arises from the set $T_i$ by intersecting
$T_i$ with a half-space as follows.
We first define a certain scaling of $F$, say $\bar{F}$,
such that $\bar{F} \subset F$.
Next we employ the Linear Integer Optimization Oracle.
If $\conv(\{y\},\bar{F}) \cap \Z^n = \emptyset$, then
Lemma~\ref{covering2} implies that we either find a
point $z \in T_i\cap K_\varepsilon \cap \Z^n$ by solving a constant number of lower-dimensional problems, or we know
that $T_i\cap K_0\cap \Z^n=\emptyset$.
On the other hand, if $\conv(\{y\},\bar{F}) \cap \Z^n \neq
\emptyset$, then we  compute a point $x^\star
\in\conv(\{y\},\bar{F}) \cap \Z^n$ closest to $y$
with respect to the normal vector of $\aff(F)$.
Let $H^\star$ be the hyperplane parallel to $\aff(F)$ and
passing through $x^\star$.
Then we use the Separating Hyperplane Oracle to determine a
$(n-2)$-dimensional hyperplane $S^\star$ in $H^\star$
separating $x^\star$ from the level set $H^\star \cap K_{\frac{\varepsilon}{2}}$.
In turn, $S^\star$ is lifted to the
$(n-1)$-dimensional hyperplane $S := \aff(\{y\},S^\star)$.
Let $S^+$ be the half-space containing $H^\star \cap K_{\frac{\varepsilon}{2}}$ and
having $S$ as its boundary. 
We then define $T_{i+1} := T_i \cap S^+$.
Lemma~\ref{cutting} guarantees a sufficient decrease of
the volume of $T_{i+1}$ with respect to $T_i$.
It remains to check the integer points between $H^\star$ and
the hyperplane parallel to $H^\star$ and passing through
$y$.
For this, we employ Lemma~\ref{covering2} again.

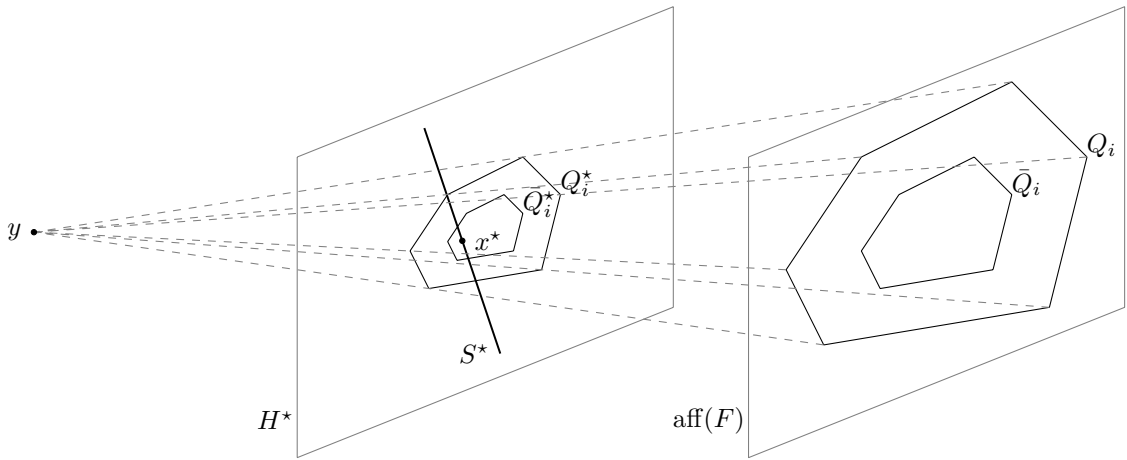
\begin{figure}[ht]
\centering
\begin{tikzpicture}[scale=1]
\draw [color=gray]{(-2.5,-3) -- (-2.5,1) (-2.5,1) -- (2.5,3) (2.5,3) -- (2.5,-1) (2.5,-1) -- (-2.5,-3)
	(3.5,-3) -- (3.5,1) (3.5,1) -- (8.5,3) (8.5,3) -- (8.5,-1) (8.5,-1) -- (3.5,-3)};	   
\draw {(4,-0.5) -- (5,1) (5,1) -- (7,2) (7,2) -- (8,1) (8,1) -- (7.5,-1) (7.5,-1) -- (4.5,-1.5) (4.5,-1.5) -- (4,-0.5)
	(5,-0.25) -- (5.5,0.5) (5.5,0.5) -- (6.5,1) (6.5,1) -- (7,0.5) (7,0.5) -- (6.75,-0.5) (6.75,-0.5) -- (5.25,-0.75) (5.25,-0.75) -- (5,-0.25)
	(-1,-0.25) -- (-0.5,0.5) (-0.5,0.5) -- (0.5,1) (0.5,1) -- (1,0.5) (1,0.5) -- (0.75,-0.5) (0.75,-0.5) -- (-0.75,-0.75) (-0.75,-0.75) -- (-1,-0.25)
	(-0.5,-0.125) -- (-0.25,0.25) (-0.25,0.25) -- (0.25,0.5) (0.25,0.5) -- (0.5,0.25) (0.5,0.25) -- (0.375,-0.25) (0.375,-0.25) -- (-0.375,-0.375) (-0.375,-0.375) -- (-0.5,-0.125)};
\draw [line width=0.75pt]{( -0.81,1.385)--(-0.305,-0.115) (-0.305,-0.115) -- (0.2,-1.615)};
\draw [dashed,color=gray]{(4,-0.5) -- (-6,0 ) (5,1) -- (-6,0 ) (7,2) -- (-6,0 ) (8,1) -- (-6,0 ) (7.5,-1) -- (-6,0 ) (4.5,-1.5) -- (-6,0 )};
\draw (-6,0 ) node[circle, fill=black, draw, ,inner sep=0pt, minimum width=2pt, label=left:$y$]{};	   
\draw (-2.3,-2.5 ) node[draw=none,fill=none,label=left:$H^\star$]{};
\draw (3.7,-2.5 ) node[draw=none,fill=none,label=left:$\aff(F)$]{};
\draw (-0.305,-0.115 ) node[circle, fill=black, draw, ,inner sep=0pt, minimum width=2pt, label=right:$x^\star$]{};
\draw (0.25,0.4) node[draw=none,fill=none,label=right:$\bar{Q}^\star_i$]{};
\draw (0.75,0.65) node[draw=none,fill=none,label=right:$Q^\star_i$]{};
\draw (6.75,0.65) node[draw=none,fill=none,label=right:$\bar{Q}_i$]{};
\draw (7.75,1.15) node[draw=none,fill=none,label=right:$Q_i$]{};
\draw(0.3,-1.615) node[draw=none,fill=none,label=left:$S^\star$]{};
\end{tikzpicture}
\caption{Construction of the truncated cones in the proof of
  Theorem \ref{thm:ConIntMin}.}
\label{fig:1}
\end{figure}

\begin{proof}[Proof of Theorem \ref{thm:ConIntMin}]
First, we apply the $\Delta$-Feasibility Oracle to $P=[-B,B]^n$.
If the oracle returns that there is no point in $K_0$, then $K_0 \cap \Z^n=\emptyset$.
So let us assume that the oracle returns a point $y \in K_{\frac{\varepsilon}{2}}$.

We use induction on the dimension $n$.
If $n=1$, then we just check whether
$\lfloor{y}\rfloor \in K_\varepsilon$ or
$\lceil{y}\rceil \in K_\varepsilon$.
In the following let us assume that $n \ge 2$, and that we
can solve all lower-dimensional problems.
Let $F_1,\dots,F_{2n}$ be the facets of $[-B,B]^n$.
Then
\begin{equation*}
[-B,B]^n=\bigcup_{j=1}^{2n}\conv(\{y\},F_j).
\end{equation*}
The following procedure is applied to every facet of $[-B,B]^n$.
Hence let us only consider an arbitrary facet
$F\in\{F_1,\dots,F_{ 2n}\}$.
We define
\begin{equation*}
Q_0:=F\quad\text{ and }\quad T_0:=\conv(\{y\},Q_0).
\end{equation*}
Let $\sigma$ denote the Euclidean distance of $y$ to
$\aff(F)$. 
Then the volume of $T_0$ is proportional to the
$(n-1)$-dimensional volume of $Q_0$ times $\sigma$.
More precisely,
$\vol(T_0)=\frac{\sigma}{n}\vol{}_{n-1}(Q_0)$.
In the following, we construct a sequence
$T_0\supset T_1\supset T_2 \ \dots$. 
The construction terminates either 
with a set
$T_k$ which contains an integer point $z\in K_\varepsilon$ we can find;
or 
 with the conclusion that no integer point
in $T_0 \cap K_0$ exists.
Below we show that 
either way we need to perform at most $O(\log(B))$ steps.
Moreover, we show that the iterative construction of a set
$T_{i+1}$ from $T_i$ is performed in polynomial time, and that
$(T_0\setminus T_i)\cap K_0 \cap\Z^n=\emptyset$, and $\vol(T_{i+1})\le
\left(1-2^{-n}(n-1)^{1-n}\right)\vol(T_i)$ for all $i$.
Since in each step the decrease of the volume is only
dependent on $n$, it is guaranteed that for some
$k \in O(\log(B))$ it holds that $\vol(T_{k})<\frac{1}{n!}$.
Then, due to Lemma \ref{minimizing}, it follows that
$\dim\left(T_k\cap\Z^n\right)\le n-1$, and we can easily
determine whether $T_k\cap K_0\cap\Z^n$ is empty or find a point in $T_k\cap K_\varepsilon\cap\Z^n$, by induction. 

The iterative construction is as follows.
Let $Q_i$ and $T_i$ be given.
First we define the auxiliary polytopes
\begin{equation*}
\bar{Q}_i:=\left\lbrace x\in\R^n\;\Big{|}\;x+\frac{1}{4n}(Q_i-Q_i)\subset Q_i\right\rbrace \quad\text{ and }\quad \bar{T}_i:=\conv(\{y\},\bar{Q}_i)
\end{equation*}
(see Figure~\ref{fig:1} and Remark \ref{rem:shrinking}).
Next we employ the Linear Integer Optimization Oracle to
solve the linear integer program
\begin{equation}\label{IPn}
\min h^\mathsf{T} x \text{ s.t. } x \in \bar{T}_i \cap \Z^n,
\end{equation}
where $h$ is the normal vector of $\aff(F)$ such that
$h^\mathsf{T} y < h^\mathsf{T} x$ for $x\in F$.
We distinguish two cases.

\underline{\bf Case 1}
The linear integer program \eqref{IPn} is infeasible.
Then $\bar{T}_i\cap\Z^n=\emptyset$.
By construction, we can apply Lemma \ref{covering2} to
determine whether there exists an $z \in
(T_i\setminus\bar{T_i}) \cap K_\varepsilon\cap \Z^n$ or
whether
$(T_i\setminus\bar{T_i}) \cap K_0\cap \Z^n=\emptyset$.
This requires to solve at most $k\le4^nn^{3n}$ subproblems of dimension
$n-1$.
Let these subproblems be contained in the hyperplanes $H_1,
\dots ,H_{k}$.
Then, for all $j=1, \dots ,k$, we test the
lower-dimensional sets
\begin{equation*}
T_i\cap H_j\cap K_0 \cap \Z^n
\end{equation*}
for feasibility.
By assumption of induction, all these problems can be solved in polynomial
time.

\underline{\bf Case 2}
The linear integer program (\ref{IPn}) has an optimal solution $x^\star$.
If $x^\star\in K_\varepsilon$, then we are done.
Otherwise, let $H^\star:=\{x\in\R^n\;|\;h^\mathsf{T} x=h^\mathsf{T} x^\star\}$, i.e.~the hyperplane containing $x^\star$ and being parallel to $\aff(F)$.
We define
\begin{flalign*}
&Q_i^\star:=T_i\cap H^{\star}\quad\text{ and }\quad T_i^\star:=\conv(\{y\},Q_i^\star),\\
&\bar{Q}_i^\star:=\bar{T}_i\cap H^\star\quad\text{ and }\quad\bar{T}_i^\star:=\conv(\{y\},\bar{Q}_i^\star).
\end{flalign*}
Using the Separating Hyperplane Oracle with $A=H^\star$ and $a=x^\star$,
let $S^\star\subset H^\star$ be a $(n-2)$-dimensional hyperplane containing $x^\star$ and separating $x^\star$ from $H^\star\cap K_{\frac{\varepsilon}{2}}$.
Next let $S$ denote the unique $(n-1)$-dimensional
hyperplane containing $y$ and $S^\star$, i.e.~$S := \aff(\{y\},S^\star)$.
Furthermore, let $S^+$ denote the half-space with boundary
$S$, and containing $H^\star\cap K_{\frac{\varepsilon}{2}}$.
Then, due to the convexity of the level set, we observe
\begin{equation}\label{equationXY}
\left(\left((T_i\setminus T_i^\star)\setminus H^+\right)\cap K_0\right)\subset \left(\left((T_i\setminus T_i^\star)\setminus H^+\right)\cap K_{\frac{\varepsilon}{2}}\right) =\emptyset.
\end{equation}
It remains to check for an improving integer point $z\in K_\varepsilon$ within $T_i^\star\setminus\bar{T}_i^\star$.
For that, we apply Lemma \ref{covering2} to $T_i^\star$ and
$\bar{T}_i^\star$ in the same way that we described in Case 1.
If none of the corresponding subproblems returns a point $z\in K_\varepsilon\cap\Z^n$, then together with (\ref{equationXY}) we know that
$\left(T_i\setminus S^+\right)\cap K_0\cap\Z^n=\emptyset$.
We define
\begin{equation*}
Q_{i+1}:=Q_i\cap S^+\text \quad\text{ and }\quad T_{i+1}:=T_i\cap S^+.
\end{equation*}
It holds that  $z\notin K_0$ for all $z\in(T_0\setminus T_{i+1})\cap\Z^n$.
In particular, from Lemma \ref{cutting} it follows, that
$\vol{}_{n-1}(Q_{i+1})\le(1-2^{-n}(n-1)^{1-n})\vol{}_{n-1}(Q_{i+1})$.
Hence,
$\vol(T_{i+1})\le(1-2^{-n}(n-1)^{1-n})\vol(T_i)$.
\end{proof}

\section{Extension to the mixed integer setting.}\label{sec:mixed-integer}

It is straightforward to extend Theorem~\ref{thm:ConIntMin}
to the mixed integer setting with a constant number of
integer variables $z_1, \dots, z_n$ and any number of
continuous variables $x_1, \dots, x_d$.
Simply replace any query with input $z \in \Z^n$ to the
evaluation oracle $f(\cdot)$ by a call of a $\Delta$-Feasibility Oracle applied to a fixed $z^\star \in \Z^n$ and
returning the value
$\min \{f(z^\star,x) \; | \; (z^\star,x) \in P\}$.

\section*{Acknowledgements.}

We thank David Adjiashvili and Michel Baes for helpful
discussions on the topic.
Thanks also to Matthias K\"oppe for his comments on a 
preliminary version of this manuscript.


\small
\bibliographystyle{amsplain}
\providecommand{\bysame}{\leavevmode\hbox to3em{\hrulefill}\thinspace}
\providecommand{\MR}{\relax\ifhmode\unskip\space\fi MR }
\providecommand{\MRhref}[2]{%
  \href{http://www.ams.org/mathscinet-getitem?mr=#1}{#2}
}
\providecommand{\href}[2]{#2}


\end{document}